\documentclass[11pt]{amsart}

\usepackage{enumitem}
\usepackage[utf8]{inputenc}
\usepackage[T1]{fontenc}

\usepackage[english]{babel}
\usepackage{fullpage}
\usepackage{enumitem} 
\usepackage{bbm}
\usepackage{dsfont}

\usepackage{amssymb,libertine}
\usepackage{amssymb}
\usepackage{array}
\usepackage{bbm}

\usepackage{mathtools}
\usepackage[svgnames]{xcolor}
\usepackage{hyperref}
\usepackage{comment}
\usepackage[colorinlistoftodos]{todonotes}

\usepackage{mathrsfs}
\usepackage{wrapfig}
\usepackage{color}
\usepackage{tikz}

\usepackage{graphicx}
\usepackage{epstopdf}

 \usepackage{latexsym}
 \usepackage{amsfonts} 

\usepackage{amsmath}
\usepackage{amsthm}
\usepackage[export]{adjustbox}

   \def\CC{\mathbb{C}}
   \def\HH{\mathbb{H}}
    \def\DD{\mathbb{D}}
    \def\NN{\mathbb{N}}
    
    \def\RR{\mathbb{R}}
    
    \def\Re{\mathfrak{R}}
    \def\fe{f_e}
    \def\fez{f_{e,0}}
    \newtheorem{Proposition}{Proposition}
\newtheorem{Theorem}[Proposition]{Theorem}
\newtheorem{Lemma}[Proposition]{Lemma}
\newtheorem{Definition}[Proposition]{Definition}

\newtheorem{Remark}[Proposition]{Remark}

\def\z{\noindent}  
\def\be{\begin{equation}}
\def\ee{\end{equation}}
    
\def\ge{\geqslant}
\def\le{\leqslant}

\def\bd{\begin{Definition}}
\def\ed{\end{Definition}}
\def\bt{\begin{Theorem}}
\def\et{\end{Theorem}}

\def\calC{\mathcal{C}}

 \def\bt{\begin{Remark}}
\def\et{\end{Remark}}

\def\epsilon{\varepsilon}
\def\bel{\begin{equation}\label}
\def\ee{\end{equation}}

\def\Ei\text{Ei}

\def\phi{\varphi}

\usepackage{graphicx}
\title{On the pointwise existence of Cauchy $\rm{P.V.}$ integrals}
\author{
NICHOLAS CASTILLO, OVIDIU COSTIN , KRITI SEHGAL
 }
\date{\today}

\begin{document}

\maketitle

\section{Abstract}
The Riemann-Hilbert (RH) approach, whose origins can be traced back to Riemann's PhD thesis, is well known to be far-reaching. It provides a general framework for expressing solutions of integrable problems such as ODEs or PDEs.  Its generalization concerning monodromy groups of Fuchsian systems is one of Hilbert's 23 problems,  cf. \cite{Both}.

In this paper we extend a basic result in scalar RH theory, the Sokhotski-Plemelj formula.  Classically, this formula is derived under assumptions of H\"older continuity, although it also holds true under weaker, Dini,  continuity conditions. Dini continuity is still too restrictive. We prove the  Sokhotski-Plemelj formula under weaker assumptions, namely continuity at a point and an $L^1$ condition.

Furthermore, we provide sufficient conditions for the existence of the Cauchy $\rm{P.V.}$ integral which are in a precise sense also necessary: weakening them runs into obstructions of a mathematical foundations nature.

\tableofcontents

%\section{Notations}
%\begin{itemize}
%    \item $\HH:=\{z\in\CC:\Im z>0\}$ is the upper half plane.
%    \item $\mathbbm{1}_E$ is the indicator function supported by the set $E$.
%\end{itemize}

\section{Preliminaries}
\subsection*{Notations:} 
\begin{itemize}
    \item $\HH:=\{z\in\CC:\Im z>0\}$ is the upper half plane, where $\Im$ denotes the imaginary part of $z$.
    \item $\mathbbm{1}_E$ is the indicator function supported by the set $E$.
      \item   For a function $f$ defined on a symmetric interval around $0$,       $f_o$ and $f_e$ denote its odd and even parts, respectively: $2f_o(t)=f(t)-f(-t)$, $f_e(t)=f(t)-f_o(t)$.
\end{itemize}

We remind the reader of the definition of the Cauchy ${\rm{P.V.}}$ integral on an interval. 
\begin{Definition}[Cauchy principal value integral]

Let $f$ be defined on $[a,b]$ and let $\tau \in (a,b)$. The Cauchy principal value integral of  $f$ based at $\tau$ is defined as 
\begin{equation}\label{PVdef}
    {\rm{P.V.}}\int_{a}^b\frac{f(t)}{t-\tau}dt\coloneqq \lim_{ \delta\to 0^+}\left(\int_{a}^{\tau-
    \delta}\frac{f(t)}{t-\tau}dt+\int_{\tau+ \delta}^{b}\frac{f(t)}{t-\tau}dt\right)
\end{equation}
when the limit exists.
\end{Definition}

This definition extends naturally to curves, see \cite{AblFok}.  We normalize the problem so that the point of interest is $\tau = 0$.
Some well known  properties of the Cauchy ${\rm{P.V.}}$ integral that we use below are \cite{Fol}:
\begin{enumerate}
\item   If $A$ is any compact subset of $\RR$ which does not contain a neighborhood of zero and $f\in L^1(\RR)$ then

\begin{equation}
    {\rm{P.V.}}\int_A\frac{f(t)}{t}dt=\int_A\frac{f(t)}{t}dt.
\end{equation}

\item  If $A$ and $B$ are compact sets containing a neighborhood of $0$, $A \subset B$, and $f\in L^1(B)$, then
\begin{equation}\label{PVEquiv}
    {\rm{P.V.}}\int_B \frac{f(t)}{t}dt=\int\limits_{B\setminus A}\frac{f(t)}{t}dt+{\rm{P.V.}}\int_A \frac{f(t)}{t}dt.
\end{equation}
\end{enumerate}
Except for positivity, P.V. integrals behave like usual integrals.

\begin{Definition}{\rm{(Nontangential limits; see \cite{Conway})}}
\label{NonTang}
    Given a sequence $\{z_n\}_{n\in\NN}$ in $\HH$, we say  $z_n$ converges non-tangentially to $\tau \in\RR$ if there exists $\epsilon>0$ so that $\arg (z_n-\tau)\in (\epsilon,\pi-\epsilon)$ for all $n$ and $z_n\to \tau$. 
\end{Definition}

Our convention is that the range of $\arg$ is $[0,2\pi)$.

%\begin{Definition}\label{Hilb}
%    {\rm{(The Hilbert transform)}}
%Given a function $u:\RR\to \CC$ we define its Hilbert transform as 
%\begin{equation}
%    H[u](z)=\frac{1}{\pi}{\rm{P.V.}}\int_\RR\frac{u(t)}{t-z}dt
%\end{equation}
%when the ${\rm{P.V.}}$ integral exists as defined by \eqref{PVdef}; see \cite{SteinandS}.
%\end{Definition}
% {\color{red}What is the domain of $H[u]$? We should also state that for $z \in \mathbb{R}$, $H[u]$ is defined using the Cauchy principal value integral and otherwise by the regular integral.}

\subsection{Setting of Riemann-Hilbert}
A simple special case of  the classical Riemann Hilbert problem  is the following. 
 Given a smooth  simple curve $\mathcal{C}$ and a function $f$ integrable along the curve,  construct a complex analytic function $\Phi$ defined on $\CC\setminus \mathcal{C}$ whose jump across the curve $\mathcal{C}$ is $f$, and which decays at infinity.   More precisely,
\begin{equation}\label{eq:5}
    \Phi^+(\tau)-\Phi^-(\tau)=f(\tau), \quad \text{for all} \quad  \tau\in \mathcal{C}
\end{equation}
where $\displaystyle \Phi^\pm(\tau)$ 
is the limit of $\Phi(z_n)$ as $z_n \to \tau \in \mathcal{C}$ and  $\{z_n\}_{n\in\NN}\subset \mathcal{D}^\pm $, and $\displaystyle \lim_{\substack{|z|\to \infty\\z\not\in \mathcal{C} }} \Phi(z)=0$.  Here $\mathcal{D}^+\,\,(\mathcal{D}^-, \text{resp.})$ denotes the points to the left (right, resp.) of the curve according to its parameterization.  For smooth curves, it is easy to reduce this problem to that in which the curve is a line segment.   In this work we take  $\mathcal{C}=[a,b]$,  $a,b\in\RR$ \footnote{ The extension to smooth curves requires some further estimates, see section \ref{Plem1}.} (in which case $\mathcal{D}^+\subset \HH$)  and $f\in L^1([a,b])$. The solution to \eqref{eq:5} is \cite{AblFok}:
 \begin{equation}\label{PhiDef}
    \Phi\left(z\right)=\frac{1}{2\pi i}\int_a^b\frac{f(t)}{t-z}dt.
  \end{equation}
 The Sokhotski-Plemelj formula additionally provides the two lateral limits of $\Phi$:
\begin{equation}\label{eq:lat-lim}
    \lim_{\substack{ z_n\to \tau \\z_n\in \mathcal{D}^{\pm}}} \Phi(z_n)=\pm\frac{1}{2}f(\tau) +\frac{1}{2 \pi i } {\rm{P.V.}}\int_a^b \frac{f(t)}{t-\tau}dt,
\end{equation}
The usual assumption under which this formula is derived is that of H\"older continuity of $f$, although in similar settings (the question of harmonic conjugates) the same conclusion is reached under the weaker condition of Dini continuity of $f$ \cite{BAF}.    We show in Theorem \ref{Plemeljthm} that   even weaker conditions suffice for \eqref{eq:lat-lim} to hold at the point $\tau$. These are: continuity of $f$ at $\tau$ and, denoting  $g(x):=f(x+\tau)$, $\displaystyle  \frac{g_o(t)}{t} \in L^1([a-\tau,b-\tau])$. (It is easy to see that Dini continuity implies the conditions above but not the other way around.)

 We then show, conversely, that  continuity of $f$ at $\tau$ is needed, and that  failure of the condition  on  $g_o$ above creates obstructions stemming from mathematical foundations to a well-defined  PV.

\section{Results:  Sufficient Conditions}
In the following, $a<\tau<b$,  $f:[a,b]\to \CC$ is in $L^1$, and $g$ is defined on $[a-\tau,b-\tau]$ by  $g(x):=f(x+\tau)$.
\begin{Proposition}
    \label{necessaryandsufficient} 
  \label{Suff} Assume that $g_o(t)/t\in L^1\left([a-\tau,b-\tau]\right)$. Then, the Cauchy principal value integral 
        \begin{equation}\label{PVexist}
            {\rm{P.V.}}\int_{a}^b\frac{f(t)}{t - \tau}dt
        \end{equation}
        exists. 
    
%    Extensions of the domain $f_o(x)/x \in L^1 \left([a, b]\right)$ for the P.V. integral run into obstructions of a mathematical foundations nature. This condition is also necessary, see §?4.1 for precise details.

     \end{Proposition}
{\bf Note.} In the proofs below  we  arrange by changes of variables that that $\tau = 0$, $a<-1$ and $b>1$.
\begin{proof}
This is straightforward:  By symmetry, $\displaystyle {\rm{P.V.}}\int_{-1}^1\frac{f_e(t)}{t}dt = \lim_{\delta\to 0^+}\left( \int_{-1}^{-\delta}\frac{f_e(t)}{t}dt + \int_{\delta}^{1}\frac{f_e(t)}{t}dt \right) = 0$. 
 Also, since $f_o(t)/t\in L^1\left([-1,1]\right)$, we simply have  $\displaystyle {\rm{P.V.}} \int_{-1}^1\frac{f_o(t)}{t}dt= \int_{-1}^1\frac{f_o(t)}{t}dt$. Therefore,
\begin{equation*}
\displaystyle {\rm{P.V.}} \int_{-1}^1\frac{f(t)}{t}dt = \displaystyle {\rm{P.V.}} \int_{-1}^1\frac{f_e(t)}{t}dt + \displaystyle {\rm{P.V.}} \int_{-1}^1\frac{f_o(t)}{t}dt = \int_{-1}^1\frac{f_o(t)}{t}dt
\end{equation*}
 and the conclusion follows.
\end{proof}

\begin{Theorem}[Sokhotski-Plemelj theorem]\label{Plemeljthm} Assume that (a) the conditions of Proposition \ref{necessaryandsufficient} hold, (b)  $f$ is continuous at $\tau$, (c)  the sequence $\{z_n\}_n$ is contained either in $\mathcal{D}^+$ or in $\mathcal{D}^-$, and (d) $\{z_n\}_n$  converges non-tangentially to $\tau$. Then, we have
        \begin{equation}\label{ThmPlem}
             \lim_{\substack{ z_n\to \tau\\z_n\in \mathcal{D}^{\pm}}} \Phi\left(z_n\right)=\pm \frac{1}{2}f(\tau)+\frac{1}{2\pi i}{\rm{P.V.}}\int_{a}^b\frac{f(t)}{t-\tau}dt.
        \end{equation}
  \end{Theorem} 
  
  \begin{proof} 

    {\bf Step 1: Reduction to the case  $\tau=0$, $[a,b]=[-1,1]$.}  As noted after Proposition \ref{necessaryandsufficient} we arrange that $\tau=0$ and $(a,b)\supset[-1,1]$. The results we are aiming for are local and it is routine to further reduce the analysis to the case when the interval equals $[-1,1]$.  For, letting  $I^c=[a,b]\setminus[-1,1]$ and choosing  $n_0$ so that, for $n>n_0$  $|z_n|<\tfrac12$, we have, for all $t\in I^c$ and $n>n_0$,  $|f(t)||t-z_n|^{-1}\le 2|f(t)|$. Dominated convergence implies that
$$\lim_{z_n\to 0}\int_{I^c}\frac{f(t)}{t-z_n}dt=\int_{I^c}\frac{f(t)}{t}dt\left(={\rm P.V.} \int_{I^c}\frac{f(t)}{t}dt\right)  $$
    
{\bf  Step 2: The limits on the interval $[-1,1]$.}   
For $f_o$ we have
\begin{equation}\label{OddPart}
    \lim_{n\to\infty}\int_{-1}^1\frac{f_o(t)}{t-z_n}dt 
    \end{equation}
and, by symmetry,
\begin{equation}
    \int_{-1}^1\frac{f_o(t)}{t-z_n}dt=\int_{-1}^1\frac{f_o(t)}{t+z_n}dt.
\end{equation}
%from which we obtain 
It follows that
\begin{multline}\label{OddInt}
    \int_{-1}^1\frac{f_o(t)}{t-z_n}dt=\frac{1}{2}\left(\int_{-1}^1\frac{f_o(t)}{t+z_n}dt+ \int_{-1}^1\frac{f_o(t)}{t-z_n}dt\right)\\
    =\int_{-1}^1\frac{tf_o(t)}{t^2-z_n^2}dt=2\int_0^1\frac{tf_o(t)}{t^2-z_n^2}dt=2\int_0^1\frac{g(t)}{1-z_n^2/t^2}dt
  \end{multline}
  where $g(t)=f_o(t)/t$.  Define the closed sector 
\begin{equation}
    S_\epsilon=\{\zeta \in\CC:\arg\zeta\in[2\epsilon,2\pi-2\epsilon]\}.
  \end{equation}
 Simple estimates show that the distance $d(S_\epsilon,1)$ is $\sin(2\epsilon)$.  The assumptions of the theorem and Definition \ref{NonTang} imply $\arg z_n^2\in(2\epsilon,2\pi-2\epsilon)$; hence, for each $z_n$, $\{z_n^2/t^2:t\in(0,1]\}\subset S_{\epsilon}$ and
\begin{equation}\label{SeqEst}
    \left|1-z_n^2/t^2\right|\geq \sin(2\epsilon)\quad \text{for all} \quad t\in(0,1], n\in \NN.
\end{equation}
Hence we have $\left|g(t)/(1-z_n^2/t^2)\right|\leq|g(t)|/\sin(2\epsilon)$.  Since by assumption $g(t)=f_o(t)/t\in L^1\left([-1,1]\right)$ the dominated convergence theorem applies and we get 
\begin{equation}\label{FinalOdd}
    \lim_{n\to\infty}2\int_0^1\frac{g(t)}{1-z_n^2/t^2}dt=2\int_{0}^1g(t)dt=2\int_0^1 \frac{f_o(t)}{t}dt={\rm{P.V.}}\int_{-1}^1\frac{f(t)}{t}dt
  \end{equation}
  The last equality in \eqref{FinalOdd} follows from linearity of $ {\rm{\text{P.V.}}}$ and the fact that $ \displaystyle {\rm{\text{P.V.}}}  \int_{-1}^1 \fe(t) t^{-1}dt=0$.

For the integral involving $f_e$, we write $f_e(t)=f_{e,0}(t)+f_e(0)$.   
\begin{equation}\label{EvenPart}
  \int_{-1}^1\frac{f_e(t)}{t-z_n}dt=\int_{-1}^1\frac{\fez(t)}{t-z_n}dt+\int_{-1}^1\frac{f_e(0)}{t-z_n}dt
    \end{equation}
For the first term on the right side of \eqref{EvenPart} we proceed as for \eqref{OddInt} and we get, using the fact that $f_{e,0}$ is even,
\begin{equation}\label{EvenInt}
    \int_{-1}^1\frac{\fez(t)}{t-z_n}dt = z_n\int_{-1}^1 \frac{\fez(t)}{t^2-z_n^2}dt.
\end{equation}
Since $f_{e,0}$ is continuous at the origin,  there is a $\delta>0$ such that
  \begin{equation}
    \label{eq:boundM}
    \sup_{[-\delta,\delta]}|f|=M_\delta<\infty.
  \end{equation}
  Splitting the domain of integration in \eqref{EvenInt} we get\begin{multline}
    \left|z_n\int_{-1}^1 \frac{\fez(t)}{t^2-z_n^2}dt\right|\leq |z_n|\left( \int_{-1}^{-\delta} \frac{|\fez(t)|}{|t^2-z_n^2|}dt+\int_{-\delta}^{\delta} \frac{|\fez(t)|}{|t^2-z_n^2|}dt+\int_{\delta}^{1} \frac{|\fez(t)|}{|t^2-z_n^2|}dt \right)\\
    =|z_n|\left(\int_{-\delta}^{\delta} \frac{|\fez(t)|}{|t^2-z_n^2|}dt+2\int_{\delta}^{1} \frac{|\fez(t)|}{|t^2-z_n^2|}dt\right).
\end{multline}
We used symmetry to combine the two integrals above.  We first estimate the integral over $[\delta,1]$.   Estimates similar to those leading to  \eqref{SeqEst} imply  $|t^2-z_n^2|\geq \delta^2 \sin(2\epsilon)$ for all  $t\in[\delta,1]$ which, in turn, give
\begin{equation}
    2|z_n|\int_{\delta}^{1} \frac{|\fez(t)|}{|t^2-z_n^2|}dt\leq \frac{2|z_n|}{\delta^2 \sin(2\epsilon)}\int_{\delta}^{1} |\fez(t)|dt\leq 2|z_n|\frac{\|\fez\|_1}{\delta^2 \sin(2\epsilon)}\to 0, \quad n\to \infty
\end{equation}
where $\|\fez\|_1$ is the $L^1$ norm of $f_{e,0}$ on $[-1,1]$.
 We make the change of variables $u=t/|z_n|$ in the integral over $[-\delta,\delta]$ and obtain
\begin{equation} \label{EvenChVar}
    |z_n|\int\limits_{-\delta/|z_n|}^{\delta/|z_n|} \frac{\left|\fez\left(u|z_n|\right)\right|}{\left||z_n|^2u^2-|z_n|^2e^{2i\varphi_n}\right|}|z_n|du=\int\limits_{-\delta/|z_n|}^{\delta/|z_n|} \frac{\left|\fez\left(u|z_n|\right)\right|}{\left|u^2-e^{2i\varphi_n}\right|}du
\end{equation}
where $\varphi_n=\arg z_n$.  We choose  $n$ large enough so that $|z_n|/\delta<1/2$ and split the domain of integration in \eqref{EvenChVar} into $\left[-\delta/|z_n|,-2\right]\cup[-2,2]\cup  \left[2,\delta/|z_n|\right]$.  Again, reasoning as for \eqref{SeqEst} we get
\begin{equation}\label{ThrePtEst}
    \left|u^2-e^{2i\varphi_n}\right|\geq \sin (2\epsilon) \quad \text{for} \quad u\in[-2,2],n\in \NN.
\end{equation}
  The dominated convergence theorem now applies yielding

\begin{equation}
    \lim_{n\to \infty}\int_{-2}^2\frac{\fez\left(u|z_n|\right)}{u^2-e^{2i\varphi_n}}du=0
\end{equation}

%\begin{equation}\label{MiddleEvenInt}
   % \left|\int_{-2}^{2} \frac{\fez\left(u|z_n|\right)}{u^2-e^{2i\varphi_n}}du\right|\leq \int_{-2}^{2} \frac{\left|\fez\left(u|z_n|\right)\right|}{|\sin (2\epsilon)|}du\leq \frac{\|f\|_1}{|\sin (2\epsilon)|}
%\end{equation}

 The integrals over the remaining two intervals can be combined due to symmetry:
 \begin{equation}\label{CompInt}
    \int\limits_{-\delta/|z_n|}^{-2} \frac{\left|\fez\left(u|z_n|\right)\right|}{\left|u^2-e^{2i\varphi_n}\right|}du+\int\limits_{2}^{\delta/|z_n|} \frac{\left|\fez\left(u|z_n|\right)\right|}{\left|u^2-e^{2i\varphi_n}\right|}du=2\int\limits_{2}^{\delta/|z_n|} \frac{\left|\fez\left(u|z_n|\right)\right|}{\left|u^2-e^{2i\varphi_n}\right|}du
 \end{equation}
 To apply again dominated convergence, we write 

 \begin{equation}\label{OuterEven}
   \int\limits_{2}^{\delta/|z_n|} \frac{\left|\fez\left(u|z_n|\right)\right|}{\left|u^2-e^{2i\varphi_n}\right|}du=\int\limits_{2}^{\infty} \mathbbm{1}_{[2,\delta/|z_n|]}\frac{\left|\fez\left(u|z_n|\right)\right|}{\left|u^2-e^{2i\varphi_n}\right|}du
\end{equation}
  Furthermore, it is easy to see that $\left|u^2-e^{2i\varphi_n}\right|\geq u^2-1\geq \frac{3}{5}(u^2+1)$ for $|u|\geq2$. Combining this estimate with \eqref{eq:boundM} we get, for $ u\in[2,\infty)$ and $n\in \NN$,
\begin{equation}\label{IntegrandUpBnd}
\mathbbm{1}_{[2,\delta/|z_n|]}\frac{\left|\fez\left(u|z_n|\right)\right|}{\left|u^2-e^{2i\varphi_n}\right|} \leq \frac{5}{3} \frac{M_\delta}{u^2+1} \in L^1([2,\infty)).
\end{equation}
  Clearly, $\displaystyle\mathbbm{1}_{[2,\delta/|z_n|]} \frac{\fez\left(u|z_n|\right)}{u^2-e^{2i\varphi_n}}\to 0$ pointwise on $[2,\infty)$. Thus, the dominated convergence theorem implies that the left side of \eqref{OuterEven} converges to $0$ as $n\to \infty$. Therefore, the first term on the right side of \eqref{EvenPart} tends to $0$ in the limit. \\

 The second term on the right side of \eqref{EvenPart} can be easily calculated by residues.  We deform $[-1,1]$ into the oriented semicircle $\gamma=\{ e^{i\theta}:\theta\in[-\pi,0]\}$ (recall that $\Im z_n>0$ for all $n$). Over the deformed contour, the dominated convergence theorem implies 
\begin{equation}\label{ContDef}
    \lim_{n\to\infty}\int_{\gamma}\frac{f_e(0)}{t-z_n}dt=f_e(0)\int_{\gamma}\frac{1}{t}dt=\pi i f_e(0)
\end{equation}
Combining the above results we see 
\begin{equation}
    \lim_{n\to\infty}\int_{-1}^1\frac{f(t)}{t-z_n}dt=\pi i f(0)+ {\rm{P.V.}}\int_{-1}^1\frac{f(t)}{t}dt
\end{equation}
from  which \eqref{eq:lat-lim} and the equivalence of ${\rm{P.V.}}$ integrals over $[-1,1]$ and $[a,b]$ established by \eqref{PVEquiv} imply \eqref{ThmPlem}.
\end{proof}

\begin{Remark}{\rm 
    Plemelj's formulas are often used along curves in the complex domain, especially when solving Riemann-Hilbert problems \cite{AblFok}. Our results should extend to the case of simple  smooth curves, but this extension needs some further work that we omit here. 
}\end{Remark}

\section{Results: Necessary Conditions}

\subsection{Necessary conditions for ${\rm{P.V.}}$ integral to be defined:
%The maximal domain of definition  P.V. integrals
}\label{MaxDom}
For $0<x<1$, we write 
\begin{equation}\label{PVeo}
    {\rm{P.V.}}\int_{0}^1\frac{f_o(t)}{t}dt = {\rm{P.V.}}\int_{0}^x\frac{f_o(t)}{t}dt  + {\rm{P.V.}}\int_{x}^1\frac{f_o(t)}{t}dt. 
\end{equation}
The last term on the right side of \eqref{PVeo} is just an ordinary integral. So, the existence of $\displaystyle{\rm{P.V.}}\int_{0}^1\frac{f_o(t)}{t}dt$ and hence $\displaystyle {\rm{P.V.}}\int_{-1}^1\frac{f(t)}{t}dt$ depends on the existence of $\displaystyle{\rm{P.V.}}\int_{0}^x\frac{f_o(t)}{t}dt$ which takes the form of the operators considered in \cite{CosFreid}. We note some properties of such an integral that stem from its intended use. On its classical functional domain of definition it is linear, nonnegative, it is an antiderivative of $f$,  and for a given $x$ the values of the operator depend only on $f$ restricted to $[0,x]$. The question is that of the possibility of extending its domain of definition, while preserving these fundamental properties,  to a wider space of functions than those such that $t^{-1}f_o\in L^1$.

This question, as well as some more general ones,  was thoroughly investigated in \cite{CosFreid}. We summarize below the notations in that paper and the result that is relevant here.
\begin{Definition}
    We define the germ equivalence relation,
    \begin{equation}
        f\sim_0 g \quad\text{if and only if }  \quad \exists \quad a>0\quad \text{s.t.} \quad f=g\,\, \text{on}\,\, (0,a].
    \end{equation}
\end{Definition}
 \begin{Definition}
     An operator A is based at $0^+$ if
$$f\sim_0 g \implies A(f)\sim_0 A(g)$$
 \end{Definition}
Operators based at other points in $\RR\cup\{\infty\}$ are analogously defined.

 \begin{Definition}
     For a continuous function $f$ on $(0,1]$ we denote by $f^{\bowtie}$ the multiplicative semigroup\footnote{This is referring to the semigroup $C^\infty\left((0,1]\right)$ not the set $f^{\bowtie}$ which does not admit such structure.} generated by $f$ and all smooth functions bounded by 1 on $(0, 1]$:

     \begin{equation}
         f^{\bowtie}:=\{f h:h\in C^\infty\left((0,1]\right),\|h\|_\infty\leq 1\}
     \end{equation}
     
 \end{Definition}

In \cite{CosFreid} the authors consider sets of operators given by the following definitions:  

\begin{Definition}\label{Def5}
    For $f\in C\left((0,1]\right)$,  $Op\left(f^{\bowtie}\right)$ is the collection of operators $\mathcal{P}_0 $ from $f^{\bowtie}$ into $C^1\left((0,1]\right)$ s.t.
    \begin{enumerate}[label=\Roman*.]
        \item $\mathcal{P}_0 $ is based at $0^+$
        \item $\forall g\in f^{\bowtie}$ we have $\left(\mathcal{P}_0 g\right)'=g$ 
    \end{enumerate}

\end{Definition}

\begin{Definition}\label{Def6}
    An extension of the integral from zero is an operator $\mathcal{Q}_0 : C\left((0,1]\right)\to C^1\left((0,1]\right)$ based at $0^+$ such that $\left(\mathcal{Q}_0 f\right)'= f$, and with the additional properties (III) $\mathcal{Q}_0 $ is linear, (IV) $(\mathcal{Q}_0 (f))(x) =\int_0^x f$ for $f\in L^1\left((0,1]\right)$, and (V) if $f > 0$ then  $\exists \ a$ s.t. $\left(\mathcal{Q}_0 f \right)(x)>0$ 
    for all $x\in (0,a)$.
\end{Definition}
For the current discussion we are interested in the following integral operator:
\begin{equation}
    \left(\mathcal{PV}_0 g\right)(x)={\rm{P.V.}}\int_0^x\frac{g(t)}{t}dt=\lim_{\epsilon \to 0} \int_\epsilon^x\frac{g(t)}{t}dt
\end{equation}
The domain of $\mathcal{PV}_0$ will be a set of the form $$\left( t^{-1}g_o\right)^{\bowtie}:=\{ t^{-1}g_o h:h\in C^\infty\left((0,1]\right),\|h\|_\infty\leq 1\}$$ 
where $g_o$ is the odd component of any $g\in C\left((0,1]\right)$.  Definition \ref{Def5} of \cite{CosFreid} establishes a family of operators $Op\left(\left( t^{-1}g\right)^{\bowtie}\right)$ based at $0^+$ which act as antidifferentiation.  We easily check that $\mathcal{PV}_0\in Op\left(\left(t^{-1}g_o\right)^{\bowtie}\right) $ which follows from the definitions.  Theorem 14 \cite{CosFreid} then shows that the statement:
$$\left(\exists t^{-1}g_o \not\in L^1   \right)\left(Op\left(\left(t^{-1} g_o\right)^{\bowtie}\right)\neq \emptyset \right)$$
is not provable in ZFDC (ZF with the Axiom of Dependent Choice). Furthermore, Theorem 16 of this same work shows that there is no definition which provably in ZFC uniquely defines some element of $Op\left(\left(t^{-1} g_o\right)^{\bowtie}\right)$ with $t^{-1}g_o \not\in L^1$.  This also holds for ZFC extended by the usual large cardinal hypotheses.
Since in applications we wish to perform explicit computations i.e. solve a Riemann Hilbert problem,  any solution needs to have some definition which Theorems 14 and 16 of \cite{CosFreid} rules out. Recalling that the existence of PV was reduced to the problem above, this shows that  the assumption of integrability i.e. $\displaystyle f_o/t\in L^1\left([-1,1]\right)$ cannot be further weakened.

\subsection{Necessary condition for the Sokhotski-Plemelj formulae to be valid:}  
The continuity condition on $f$ cannot be eliminated as can be seen in a very straightforward way: let $f:[-1,1]\to \mathbb{R}$ be given by
\begin{equation}
    f(t)=\begin{cases}
    t & t\neq 0\\
    1 & t=0
\end{cases} 
\end{equation}
then the lateral limits and the $\text{P.V.}$ integrals exist, but \eqref{ThmPlem} fails.

\section{Acknowledgements}
We are extremely grateful for Dr. Maxim Yattselev's comments about our previous paper on the generalization of Sokhotski-Plemeljs formulas which led to the work in this paper.  Additionally, we are also extremely grateful for Rodica Costin's generous help and keen suggestions.   \\

The work of OC is supported in part by the U.S. National Science Foundation, Division of
Mathematical Sciences, Award NSF DMS-2206241.

\section{Supplementary Material}\label{Plem1}
We provide an alternative proof to the Sokhotski-Plemelj formulae for Dini continuous functions on a simple smooth curve. We note that this extension follows from \cite{BAF}, however, our proof relies only on elementary analysis.
\subsection{Preliminaries} We start by reviewing some basic definitions and establish the notations we use. The curve $\mathcal{C}$ in the complex plane  $\mathbb{C}$  parameterized by $\psi: [a,b] \rightarrow \mathcal{C}$ is said to be simple is $\psi(t_1) \neq \psi(t_2)$ for $t_1 \neq t_2$, except possibly at the end points. If $\psi(a) = \psi(b)$, we say that the curve $\mathcal{C}$ is closed. 
We say $\mathcal{C}$ is smooth if $\psi$ is continuously differentiable up to some specified order and has non-zero first derivative at each point. 

\begin{Definition}[Cauchy Principle Value (P.V.) for curves]
    Let $\mathcal{C}$ be a simple and smooth curve. Consider a $t_0\in\overline{\mathcal{C}}\setminus \partial \mathcal{C}$\footnote{It's essential that $t_0$ is not an endpoint of the curve, otherwise the P.V. integrals will not be defined. Here we specify the endpoints which form the set $\partial\mathcal{C}=\overline{\mathcal{C}}\setminus \mathring{\mathcal{C}}$ defined by the metric topology on $\mathcal{C}$ induced from the arc-length parameterization.} and let $\phi$ be a continuous and integrable function on $\mathcal{C}$. By $\DD_{\epsilon}(t_0)$ we denote disks of radius $\epsilon$ centered at $t_0$. The Cauchy principal value line integral $\displaystyle {\rm P.V.}\int_{\mathcal{C}} \frac{\phi(t)}{t-t_0} dt $ denotes the limit of the integrals outside $\DD_\epsilon (t_0)$ \cite{Muskhel},
\begin{equation}
  \label{eq:CauchyPV}
  \lim_{\epsilon\to 0} \int_{\mathcal{C}\cap {\DD_{\epsilon}(t_0)}^c} \frac{\phi(t)}{t-t_0} dt 
\end{equation}
when this limit exists. 
\end{Definition}

We next review Hölder and Dini continuity. Consider $X$, $Y$ be two Euclidean spaces with euclidean norm $\|.\|_X$ and $\|.\|_Y$ respectively, and $f:X \rightarrow Y$ be a function between them.
\begin{Definition}[Hölder continuous]
    $f$ satisfies Hölder condition or is Hölder-continuous if there exists $C \ge 0$ and $\alpha > 0$ such that
    \begin{equation*}
        \|f(x_1)-f(x_2)\|_Y \le C \|x_1-x_2\|_X^{\alpha}
    \end{equation*}
    for all $x$, $y$ in $X$. 
\end{Definition}
The number $\alpha$ is called the exponent of the Hölder condition. For an open subset $\Omega$ of a Euclidean space and $k \ge 0$ an integer, $C^{k,\alpha} (\Omega)$ denotes the space consisting of all those functions on $\Omega$ having continuous derivatives up through order $k$ and such that the $k^{\text{th}}$ partial derivatives are Hölder continuous with exponent $\alpha$, where $0< \alpha \le 1$. \\

The modulus of continuity is a concept used to measure the continuity of a function. It quantifies how the function's values change as its input varies. The modulus of continuity provides a bound on the maximum difference between function values within a given range of inputs. 
\begin{Definition}[Modulus of continuity]
The modulus of continuity of a bounded function $f:X \rightarrow Y$ is a non-decreasing function $\omega_f: [0,\infty) \rightarrow [0,\infty)$ defined as
\begin{equation}\label{ModulusofCont}
    \omega_f(t) := \sup\left\{ \|f(x_1)-f(x_2)\|_Y:\|x_1-x_2\|_X \leq t\right\}.
\end{equation}
\end{Definition} 

\noindent  Moduli of continuity such as $\omega_f$ find themselves at an intersection of many classical and fundamental areas of analysis. For $\alpha\in (0,1]$, a function $f$ is $\alpha$-Hölder continuous if and only if there exists a constant $M >0$ such that $\omega_f(t) \le M t^{\alpha}$ for all $t \ge 0$.
%Further generalizations yield conditions on $L^p$ functions being weakly differentiable i.e. belonging to $W^{1,p}$ again based upon the behavior of a modulus of continuity; see Theorem 3 in \textsection 5.8 \cite{Ev}.  The small $t$ behavior of $\omega_f$ can be used to study convergence of Fourier series; see \cite{Zy} Vol I, Chapter 2 \textsection 3.
\begin{Remark}
  Throughout the section we work on the metric space defined by the smooth curve $\mathcal{C}$ equipped with the Euclidean metric inherited from $\CC$. While in the class of smooth curves the arc-length metric can be equivalently used,
    and the results would be {\em mutatis-mutandis} the same, this choice simplifies the notation. 
     
\end{Remark}

\begin{Definition}[Dini continuous]
    $f$ is said to be a Dini continuous function if 
    \begin{equation*}
        \int_{0}^{1} \dfrac{\omega_f(t)}{t} dt < \infty.
    \end{equation*}
\end{Definition}

\begin{Remark}
If $f:X \rightarrow Y$ is a Dini continuous function on $X$ then it is also continuous on $X$. 
\end{Remark}
%To prove the above, assume that $f$ is a Dini continuous function on $X$ which is not continuous at $x_0 \in X$. Then there exists $\epsilon_0 > 0$ such that for each $\delta \ge 0 $, there is $x_{\delta} \in X$ with $\|x_0 - x_{\delta}\|_X<\delta$ but $\|f(x_0) - f(x_{\delta})\|_Y \ge \epsilon_0$. 
%\begin{equation*}
%    \omega_f(\delta) = \sup_{\|x_1-x_2\|_X\le \delta}\|f(x_1),f(x_2)\|_Y \ge \epsilon_0 \implies \int_{0}^{1} \dfrac{\epsilon_0}{t} dt \le \int_0^1 \dfrac{\omega_f(t)}{t} dt
%\end{equation*}
%Since the left integral is infinite, the above contradicts the fact that $f$ is Dini continuous.  
\begin{Remark}
    Dini continuity is strictly weaker than Hölder continuity.
\end{Remark}
We can easily check that the function $f$ defined by $f(x)=\mathbbm{1}_{[0,\mathrm{e}^{-3})}\log(x)^{-2}$ is Dini continuous, but for no exponent $\alpha$ is it H\"older continuous.
\newpage
\subsection{Main Results}
\begin{Theorem} \label{MainThm}
Given a simple, smooth curve $\mathcal{C} \subset \CC$ and 
 $\varphi$ a Dini continuous, integrable\footnote{\label{L1Con}We assume $\varphi \in L^1\left(\mathcal{C},d|\mu_\mathcal{C}|\right)$, where $\mu_\mathcal{C} $ is the complex measure defined by $d\mu_\mathcal{C}=\psi'(t)dt$, $\psi$ is the arc length parameterization of $\calC$, and $dt$ is the Lebesgue measure.} function defined on $\calC$. Then, for any $t\in \mathcal{C}\setminus \partial \mathcal{C}$, the following are true: 
\begin{enumerate}[label=(\roman*)]
    \item \label{exists}
    The {\rm{P.V.}} integral along $\calC$ defined by 
    \begin{equation}\label{PVintexist}
        \text{\rm{\rm{P.V.}}}\int_{\mathcal{C}}\frac{\varphi(s)}{s-t}ds
    \end{equation} 
    exists. See \eqref{eq:CauchyPV} for the definition of {\rm{P.V.}} integrals along contours.
    \item  \label{SeqCurv} Consider any sequence $\{z_n\}_{n\in\NN}$ contained in $\CC\setminus\calC$ converging to the point $t\in \mathcal{C}\setminus \partial \mathcal{C}$ from the left (right) of the curve as defined by the orientation of the curve.  Then, with the $\pm$ sign, $+$ being for the left limit and $-$ for the right limit we have the following
    \begin{equation}
        \lim_{n\to \infty} \Phi(z_n)=\Phi^{\pm}(t)
    \end{equation}
    where 
    \begin{equation}\label{PlemeljForm}
        \Phi^\pm(t)=\pm\frac{1}{2}\varphi(t)+\frac{1}{2 \pi i } \text{\rm{\rm{P.V.}}}\int_{\mathcal{C}}\frac{\varphi(s)}{s-t}ds.
    \end{equation}
\end{enumerate}

\end{Theorem}

Before the proof of Theorem \ref{MainThm} we present two helpful technical results Lemma \ref{DiniLemma} and Lemma \ref{LimitLemma}. The proof of these lemmas are provided in the next section. The first one shows an invariance of the space of Dini continuous functions under two basic operations with continuously differentiable functions.

\begin{Lemma}
\label{DiniLemma}
Let $K$ be  compact subset of our simple smooth curve $\mathcal{C} \subset \mathbb{C}$. Suppose $\phi : K \rightarrow \mathbb{C}$ is a Dini continuous function. Suppose, for a given interval $[a,b]$, $\psi : [a,b] \rightarrow K$ and $\chi : \mathcal{C} \rightarrow \mathbb{C}$ are continuously differentiable functions. Then the functions $\phi \circ \psi: [a,b] \rightarrow \mathbb{C}$ and $\chi \phi : K \rightarrow \mathbb{C}$ are Dini continuous functions.
\end{Lemma}
%{\color{red} Note for Nick and Kriti: Ovidiu suggested we could assume that the curve is compact, but, for now  we only assume that $I$ is a closed interval and $\psi$ is continuously differentiable on $I$ so that $\psi'$ is bounded. Later, if needed, we can make the assumption stronger by saying that the curve is compact. }

Next, Lemma \ref{LimitLemma} establishes the source of the residue term  $\pm \frac{1}{2}\varphi(t)$ in Plemelj's formula.  This result shows the necessity of the single sided approach to the curve.  Without loss of generality, we assume the sequence $\{z_n\}_{n\in\NN}$ from the hypothesis of Theorem \ref{MainThm} is converging to the origin and the curve $\calC$ is parameterized by arc length via \label{param}$\psi(t)=u(t)+iv(t)$.

\begin{Lemma} \label{LimitLemma}
Given $0<c<1$, a Dini continuous function $F: [-1,1] \rightarrow \mathbb{C}$, and a real sequence $\mathbf{x}\in \mathbf{c_0}(\RR)$\footnote{$\mathbf{c_0}(\RR)$ is the subspace of $\ell^\infty(\RR)$ consisting of all $\RR$ valued sequences converging to zero with the norm $\|\mathbf{x}\|_\infty = \sup_n |x_n|$.} such that $c+\|\mathbf{x}\|_\infty<1$.  Consider the sequence of functions $F_n:[-c,c] \rightarrow \mathbb{C}$ defined by 
\begin{equation*}
    F_n(w)=F(w+x_n)
\end{equation*}
which  converges pointwise to $F$ by continuity. Let $\{J_n\}_{n\in \NN}$, $J_n:[-c,c] \rightarrow \mathbb{R}$  be defined by 
\begin{equation*}
    J_n(w)=\zeta_n+G(w+\eta_n)
\end{equation*}
where $\{\zeta_n\}_{n\in\NN},\{\eta_n\}_{n\in\NN}$ are real sequences which converge to zero, and $G:[-1,1] \rightarrow \mathbb{R}$ is given by  $G(t)=(v\circ u^{-1})(t)$ where $\psi(t)=u(t)+iv(t)$. We choose  $ \eta = \{\eta_n\}_{n \in \mathbb{N}}$ so that  $c + \|\mathbf{\eta}\|_\infty <1$. We assume that:

\begin{enumerate}
     \item $u,v$ are bijections, $\psi(0) = 0$, and $\psi'(0) = 1$;
     \item $G\in C^{k,\alpha}([-c,c],\RR)$ for $k\geq 1$ and $\alpha\in(0,1]$ satisfying $G(0)=G'(0)=0$;
     \item $J_n$ possesses only one sign on $[-c,c]$ for all $ n\in \NN$. 
\end{enumerate}
Then the following are true: 
\begin{enumerate}[label=(\roman*)]
    \item \begin{equation}\label{LemmaEq}
    \lim_{c\to 0}\lim_{n\to \infty}\int _{-c}^c\frac{F_n(w)}{w-iJ_n(w)}dw=\pm\pi i F(0)
\end{equation}
The sign in \eqref{LemmaEq} follows that of $J_n$. 

\item The principal value integral 
\begin{equation}\label{LemmaGenExists}
    {\rm{P.V.}}\int_{-c}^c \frac{F(w)}{w-iJ(w)}dw=\lim_{\epsilon\to 0} \int_{[-c,c]\cap \DD_{\epsilon}^c} \frac{F(w)}{w-iJ(w)} dw
\end{equation}
exists, where the function $J$ is the pointwise limit of $J_n$.  
\item \label{limit_to_1} We are free to let $c\to 1$ in \eqref{LemmaGenExists} and hence

\begin{equation} \label{Ceq1}
     {\rm{P.V.}}\int_{-1}^1\frac{F(w)}{w-iJ(w)}dw=\lim_{\epsilon\to 0} \int_{[-1,1]\cap \DD_{\epsilon}^c} \frac{F(w)}{w-iJ(w)} dw
\end{equation}
exists.

\end{enumerate}
\end{Lemma}

\begin{proof}[Proof of Theorem \ref{MainThm}]
$ $\newline
\indent The theorem is proved first for the case when $\mathcal{C}=[-1,1]$ and  $t=0$ and then the proof for a general smooth curve reduces to this case. \\
Case-1: $\mathcal{C}=[-1,1]$ and  $t=0$. Here $\psi: [-1,1] \rightarrow [-1,1]$ with $\psi (x) = x.$
\begin{enumerate}[label=(\roman*)]
    \item The existence of $\displaystyle {\rm{P.V.}} \int_{-1}^{1} \dfrac{\phi(s)}{s-0}ds$ is an immediate consequence of Lemma \ref{LimitLemma} \ref{limit_to_1} with $J(w)=0$.

    %it sufficient to prove the existence of P.V. $\int_{-c}^{c} \dfrac{\phi(s)}{s-0}ds$ for some $0<c<1$. For proving the latter, we apply \eqref{LemmaGenExists} of Lemma \ref{LimitLemma}.
    %Take $F(w)=\varphi(w)$, $\zeta_n\ = 0$ $\forall$ $n$. Since the curve $\mathcal{C} = [-1,1]$, therefore, the parametrization $\psi (t) = u(t) + i v(t)$ satisfies $u(t) = t$ and $v(t) = 0$ for $t \in [-1,1]$. Thus, $J_n(w) = 0 + v(u^{-1}(w + \eta_n)) = 0$ for all $n$, therefore, $J(w) = 0$ for $w \in [-c,c]$. Hence, P.V. $\int_{-c}^{c} \dfrac{\phi(s)}{s-0}ds$ exists.
    
    \item Take a sequence $\{z_n\}_{n\in\NN}$ confined to the upper half plane converging to $t = 0$. Write $z_n=x_n+i y_n$ such that $y_n>0$ for all $n\in\NN$. This means that our sequence $\{ z_n \} _{n \in \mathbb{N}}$ lies to the left of $\mathcal{C}$. We wish to compute 
    \begin{equation}
        \lim_{n\to \infty}\Phi(z_n)= \lim_{n\to \infty}\frac{1}{2\pi i} \int_{-1}^1\frac{\varphi(s)}{s-z_n}ds.
    \end{equation}
    Let $0<c <1 $, make the change of variable $u=s-x_n$ and decompose the domain of integration into $[-1-x_n,-c]\cup[-c,c]\cup[c,1-x_n]$.
    \begin{align*}
        \int_{-1}^{1} \dfrac{\phi(s)}{s - z_n} ds &= \int_{-1-x_n}^{1-x_n} \dfrac{\phi(u + x_n)}{u+x_n - z_n} du = \int_{-1-x_n}^{1-x_n} \dfrac{\phi(u+x_n)}{u - i y_n} du \\
        &= \int_{-1-x_n}^{-c} \dfrac{\phi(u+x_n)}{u - i y_n} du + \int_{-c}^{c} \dfrac{\phi(u+x_n)}{u - i y_n} du + \int_{c}^{1-x_n} \dfrac{\phi(u+x_n)}{u - i y_n} du
    \end{align*}
    In the above, $z_n \rightarrow 0$ implies $x_n \rightarrow 0$, so, for large enough $n$  we have that $-1-x_n < -c$ and $c< 1-x_n$. Note that, by dominated convergence,     
    \begin{equation}\label{NSlim}
       \lim_{n\to \infty} \left (\int_{-1-x_n}^{-c}\frac{\varphi(u+x_n)}{u-iy_n}du+\int_{c}^{1-x_n}\frac{\varphi(u+x_n)}{u-iy_n}du\right )=\int_{-1}^{-c}\frac{\varphi(u)}{u}du+\int_c^1\frac{\varphi(u)}{u}du.
    \end{equation}
    We now let $c\to0$. 
    \begin{align*}
        \lim_{n\to \infty} \Phi(z_n) &= \lim_{c \to 0} \lim_{n\to \infty}\Phi(z_n) \\ 
        &= \dfrac{1}{2 \pi i} \lim_{c \to 0} \lim_{n\to \infty} \left( \int_{-1-x_n}^{-c} \dfrac{\phi(u+x_n)}{u - i y_n} du  + \int_{c}^{1-x_n} \dfrac{\phi(u+x_n)}{u - i y_n} du \right) \\ & +  \dfrac{1}{2 \pi i} \lim_{c \to 0} \lim_{n\to \infty} \int_{-c}^{c} \dfrac{\phi(u+x_n)}{u - i y_n} du \\
        &= \dfrac{1}{2 \pi i} \lim_{c \to 0} \left( \int_{-1}^{-c}\frac{\varphi(u)}{u}du+\int_c^1\frac{\varphi(u)}{u}du \right) + \dfrac{1}{2 \pi i} \lim_{c \to 0} \lim_{n\to \infty} \int_{-c}^{c} \dfrac{\phi(u+x_n)}{u - i y_n} du \\
        &= \dfrac{1}{2 \pi i}{\rm{P.V.}} \int_{-1}^{1} \dfrac{\phi(s)}{s} ds + \dfrac{1}{2 \pi i} \pi i \phi(0)
    \end{align*}
    For the last term (the integral over the interval $[-c,c]$), we appeal to Lemma \ref{LimitLemma} with $F(w) = \phi(w)$, $F_n(w)=\varphi(w+x_n)$, $\zeta_n = y_n$. This implies  $J_n(w)=\zeta_n + v(u^{-1}(w + \eta_n)) = y_n$ for any sequence $\eta_n$ since $v=0$; hence the residue term. Thus,
    \be
    \Phi^+(0)=\lim_{n\to \infty} \Phi(z_n)=\frac{1}{2}\varphi(0)+\frac{1}{2\pi i}{\rm{P.V.}}\int_{-1}^1 \frac{\varphi(u)}{u}du
    \ee
    The same argument applies to $\Phi^-(0)$, where the minus sign on the residue comes from the reversal of sign of $J_n$; see Lemma \ref{LimitLemma}. 
  
\end{enumerate}
Case-2: $\mathcal{C}$ is any general smooth curve \label{Gen} \\ 
\indent Plemelj's formulas are often used along curves in the complex domain.
For a general simple, smooth curve $\mathcal{C}$ we use the fact that $\overline{\mathcal{C}}\setminus \partial \mathcal{C}$  has the structure of a one dimensional, smooth, oriented manifold without boundary.  We note that orientatability follows from simplicity.  Given $t \in \overline{\mathcal{C}}\setminus \partial \mathcal{C}$, which we assume is equal to $0$ since we can always change coordinates by a translation.  Furthermore, we choose a parameterization $\psi:(-1,1)\to \overline{\mathcal{C}}\setminus \partial \mathcal{C}$ such that $\psi(0)=0$ and $\psi'(0)=1$.\\
A short calculation shows the equivalence between our definition of P.V. line integrals \eqref{eq:CauchyPV} and the standard P.V. integral on the real line when using a parameterization of $\mathcal{C}$
\begin{multline}\label{GenPV}
 \text{\rm{\rm{P.V.}}}\int_{\mathcal{C}}\frac{\varphi(s)}{s}ds=\lim_{\epsilon\to 0} \int_{\mathcal{C}\cap \DD_{\epsilon}^c} \frac{\phi(s)}{s} ds\\
     =\lim_{\epsilon\to 0}\left( \int_{-1}^{-c_1(\epsilon)}\frac{(\varphi\circ \psi)(\tau)}{\psi(\tau)}\psi'(\tau)d\tau +\int_{c_2(\epsilon)}^{1}\frac{(\varphi\circ \psi)(\tau)}{\psi(\tau)}\psi'(\tau)d \tau \right)
\end{multline}
where $-c_1(\epsilon),c_2(\epsilon)\in \psi^{-1}(\mathcal{C}\cap \partial \DD_{\epsilon})$ and $\epsilon$ is small enough to ensure the cardinality $|\mathcal{C}\cap \partial \DD_{\epsilon}|=2$ which is guaranteed by the smoothness of $\mathcal{C}$. Note that, $\psi(-c_1(\epsilon))$ is the point on the left of the origin in $\mathcal{C} \cap \partial \DD_{\epsilon}$ and $\psi(c_2(\epsilon))$ is on the right. Moreover, without loss of generality we assume $c_1(\epsilon),c_2(\epsilon)>0$ and $c_1(\epsilon)>c_2(\epsilon)$ then the right hand side of \eqref{GenPV} can be written 

\begin{multline}\label{ExistCurve}
    \lim_{\epsilon\to 0}\left\{\int_{-1}^{-c_1(\epsilon)}\frac{(\varphi\circ \psi)(\tau)}{\psi(\tau)}\psi'(\tau)d\tau +\int_{c_1(\epsilon)}^{1}\frac{(\varphi\circ \psi)(\tau)}{\psi(\tau)}\psi'(\tau)d \tau+\int_{c_2(\epsilon)}^{c_1(\epsilon)}\frac{(\varphi\circ \psi)(\tau)}{\psi(\tau)}\psi'(\tau)d \tau\right \}\\
    ={\rm{P.V.}}\int_{-1}^{1}\frac{(\varphi\circ \psi)(\tau)}{\psi(\tau)}\psi'(\tau)d\tau+  \lim_{\epsilon\to0}\int_{c_2(\epsilon)}^{c_1(\epsilon)}\frac{(\varphi\circ \psi)(\tau)}{\psi(\tau)}\psi'(\tau)d \tau
\end{multline}
The first term on the right hand side of \eqref{ExistCurve} can be seen to exist by writing 
$$\frac{(\varphi\circ \psi)(\tau)}{\psi(\tau)}\psi'(\tau)=\frac{(\varphi\circ \psi)(\tau)\psi'(\tau)}{\tau}\frac{\tau}{\psi(\tau)}$$
and appealing to Lemma \ref{DiniLemma} we see $(\varphi\circ \psi)(\tau)\psi'(\tau)$ is Dini.  If we can show the function $\tau/\psi(\tau)$ is continuously differentiable then again by Lemma \ref{DiniLemma} $(\varphi\circ \psi)(\tau)\psi'(\tau)\tau/\psi(\tau)$ will be Dini and an application of Lemma \ref{LimitLemma} {\it{(iii)}} with $J(\tau)=0$ proves the existence of $\displaystyle {\rm{P.V.}}\int_{-1}^{1}\frac{(\varphi\circ \psi)(\tau)}{\psi(\tau)}\psi'(\tau)d\tau$. We define 
$$h(\tau)=\begin{cases}
\frac{\tau}{\psi(\tau)} & \tau\neq 0\\
1 &  \tau=0
\end{cases}$$
Recalling that $\psi$ is the parameterization of our smooth and simple curve $\mathcal{C}$ implies that the only solution to the equation $\psi(\tau)=0$ is $\tau=0$. 
 Hence the only point of concern for $h$ is $\tau=0$ since it is a rational combination of smooth functions.  Furthermore, the assumed normalization  $\psi(0)=0$ and $\psi'(0)=1$ shows that $\lim_{\tau\to 0} h(\tau)=1$ (since this is the reciprocal of the difference quotient $(\psi(\tau)-\psi(0))/\tau$), proving the continuity of $h$ on $[-1,1]$.  To show differentiability, we differentiate away from the origin using the quotient rule and extend the derivative to $0$ by computing the difference quotient of $h$.  
 \begin{equation}\label{hPrime}
     h'(\tau)=
     \begin{cases}
          \frac{\psi(\tau)-\tau\psi'(\tau)}{\psi(\tau)^2} & \tau\neq 0\\[6pt]
          -\frac{\psi''(0)}{2} & \tau=0
     \end{cases}
 \end{equation}
 Again we see by inspection that the only point of concern in \eqref{hPrime} is $\tau=0$.  Using the normalization of $\psi$ we see that the conditions for twice applying  L'H\^{o}pital's rule are met. \footnote{After proper decomposition into real and imaginary part.}

 \begin{equation}
     \lim_{\tau \to 0} h'(\tau)=\lim_{\tau \to 0} \frac{\psi(\tau)-\tau\psi'(\tau)}{\psi(\tau)^2}=\lim_{\tau \to 0}\frac{\psi'(\tau)-\psi'(\tau)-\tau\psi''(\tau)}{2\psi'(\tau)\psi(\tau)}=\lim_{\tau \to 0}\frac{-\tau\psi''(\tau)}{2\psi'(\tau)\psi(\tau)}
 \end{equation}
 Once more we verify the conditions to use L'H\^{o}pital's rule and compute
\begin{equation}
    \lim_{\tau \to 0} h'(\tau)=\lim_{\tau \to 0}\frac{-\tau\psi''(\tau)}{2\psi'(\tau)\psi(\tau)}=\lim_{\tau \to 0}\frac{-\psi''(\tau)-\tau\psi^{(3)}(\tau)}{2\psi''(\tau)\psi(\tau)+2\psi'(\tau)^2}=-\frac{1}{2}\psi''(0).
\end{equation}

%\r{Taylor's theorem applied to $\psi^{-1}$ implies $c_1(\epsilon),c_2(\epsilon)=\epsilon+O(\epsilon^2)$ for small $\epsilon$.} Estimating the second integral

%\begin{equation}
   %\Big| \int_{c_2(\epsilon)}^{c_1(\epsilon)}\frac{(\varphi\circ \psi)(\tau)}{\psi(\tau)}\psi'(\tau)d \tau\Big| \leq \frac{K}{\epsilon}\epsilon^2
%\end{equation}
We consider $\displaystyle \int_{c_2(\epsilon)}^{c_1(\epsilon)}\frac{(\varphi\circ \psi)(\tau)}{\psi(\tau)}\psi'(\tau)d \tau$ as a function of $\epsilon$ and note that $[c_2(\epsilon),c_1(\epsilon)]$ (the region of integration) is disjoint from the origin.  Furthermore, by Dini continuity, the integrand is integrable all the way down to the origin.  Using the fundamental theorem of Lebesgue integrals \cite{Fol} the integration will evaluate to  $K(c_1(\epsilon))-K(c_2(\epsilon))$ where $K$ is an absolutely continuous function.  Taking the limit $\epsilon\to 0$ and using continuity of $c_j(\epsilon)$ tells us that 
$$\lim_{\epsilon\to0} \int_{c_2(\epsilon)}^{c_1(\epsilon)}\frac{(\varphi\circ \psi)(\tau)}{\psi(\tau)}\psi'(\tau)d \tau=0$$ 
implying existence.    

Additionally, we recover the relation 
\begin{equation}
    \text{\rm{\rm{P.V.}}}\int_{\mathcal{C}}\frac{\varphi(s)}{s}ds=\text{\rm{\rm{P.V.}}}\int_{-1}^1\frac{(\varphi\circ \psi)(\tau)}{\psi(\tau)}\psi'(\tau)d\tau
\end{equation}
from which we see that principal value integrals along curves are defined by principal value integrals over intervals.  One readily checks that this is independent of parameterization. \vspace{6pt} 

For \ref{SeqCurv} we proceed in a way similar to the earlier argument and assume without loss of generality that our parameterization is normalized so that $\psi(0)=0$ and $\psi'(0)=1$ .  Given $\{z_n\}_{n\in\NN}$ converging to $0$ from the left of the curve we first project the sequence onto the $\mathcal{C}$ along the imaginary axis.  In other words, we decompose $z_n=\lambda_n+i\mu_n$ with $\Re z_n=\Re \lambda_n$ and $\lambda_n\in\calC$.  Notice that projecting along $i\RR$ ensures $i\mu_n\in i\RR$ and furthermore, since we are assuming $\{z_n\}_{n\in\NN}$ is approaching from the left we have $\Im i\mu_n>0$ for all $n\in \NN$.   Since  $\lambda_n\in\calC$ we know there exists a unique $t_n\in[-1,1]$ such that $\psi(t_n)=\lambda_n$.  We now consider $\lim_{n\to\infty}\Phi(z_n)$ by writing it in a form analogous to the earlier argument.  First, choose $r>0$ so that the function $u(\tau)=\Re \psi(\tau)$ is invertible on the set $\psi^{-1}\big(\calC \cap \DD_r\big)$.  Such an $r$ exists since we are normalizing so that $u'(0)=1$. 

\begin{equation}
    2\pi i\Phi(z_n)=\int_{\mathcal{C}}\frac{\varphi(s)}{s-z_n}ds=\int_{\calC \cap \DD_r^c}\frac{\varphi(s)}{s-z_n}ds+\int_{\calC \cap \DD_r}\frac{\varphi(s)}{s-z_n}ds.
    \end{equation}
    %\frac{1}{2\pi i}\int_{-1}^1\frac{(\varphi\circ \psi)(\tau)}{\psi(\tau)-\psi(t_n)-i\mu_n}\psi'(\tau)d\tau
The integral over $\calC \cap \DD_r^c$ is nonsingular and the limit may be passed through the integral by dominated convergence theorem, yielding

\begin{equation}\label{GenPhin}
   2\pi i \lim_{n\to\infty}\Phi(z_n)=\int_{\calC \cap \DD_r^c}\frac{\varphi(s)}{s}ds+\lim_{n\to\infty}\int_{\calC \cap \DD_r}\frac{\varphi(s)}{s-z_n}ds.
\end{equation}

Recalling that an important aspect of the proof in the simpler case was uncovering an integral over a symmetric interval which resulted in the principal value integral.  We seek the same here, but first we rewrite the second term on the right hand side of  \eqref{GenPhin} in terms of the parameterization $\psi$. 
$ $\newline 
\indent Let $\psi^{-1}\big(\calC \cap \DD_r\big)=(-c_1,c_2)$ where $c_1,c_2>0$ from which we have 

\begin{equation}
    \int_{\calC \cap \DD_r}\frac{\varphi(s)}{s-z_n}ds=\int_{-c_1}^{c_2}\frac{(\varphi\circ\psi)(\tau)}{\psi(\tau)-z_n}\psi'(\tau)d\tau=\int_{-c_1}^{c_2}\frac{(\varphi\circ\psi)(\tau)}{\psi(\tau)-\psi(t_n)-i\mu_n}\psi'(\tau)d\tau.
\end{equation}
Writing $\psi(\tau)=u(\tau)+iv(\tau)$, and since $u$ is invertible near the origin we can make the change of variables $w=u(\tau)-u(t_n)$. 

We denote the new translated domain of integration as $E_n=[u(-c_1),u(c_2)]-u(t_n)$ and obtain 
\begin{equation} \label{GenCaseRed}
  \int_{E_n}\frac{1}{w-iJ_n(w)}H_n(w)dw.
\end{equation}
Where the sequence of functions $H_n(w)$ are defined by 
\begin{equation}
  H_n(w)=\frac{(\varphi\circ \psi)( u^{-1}(w+u(t_n)))}{u'(w+u(t_n))}\psi'(u^{-1}(w+u(t_n))).
\end{equation}
This sequence converges pointwise to the function 
\begin{equation}
    H(w)=\frac{(\varphi\circ \psi)( u^{-1}(w))}{u'(w)}\psi'(u^{-1}(w)).
\end{equation}

The term in the denominator of \eqref{GenCaseRed} is 

\begin{equation}
    J_n(w)=v(t_n)+\mu_n-v( u^{-1}(w+u(t_n))).
\end{equation}
A crucial property is that $J_n(w)>0$ for all $n\in\NN$ and all $w\in E_n$. This follows from the assumption that $\{z_n\}_n$ remains to the left of the curve for all $n$ and noticing that $v(t_n)+\mu_n=\Im z_n$ which must be larger than the imaginary part of any point along the curve for sufficiently small $r$.

%Nick rewrote this part of the argument 

%that our new parameterization has mapped a neighborhood of the curve containing the origin into the interval $E_n$.  

We pick $c>0$ depending on $r$ and $n$ such that $[-c,c]\subset E_n$ for all $n$ and write $E_n=E_n\setminus [-c,c]\cup [-c,c]$.  
Hence we have 
\begin{equation} \label{GenCaseRed2}
  \int_{E_n}\frac{1}{w-iJ_n(w)}H_n(w)dw=\int_{E_n\setminus [-c,c]}\frac{1}{w-iJ_n(w)}H_n(w)dw+\int_{-c}^c\frac{1}{w-iJ_n(w)}H_n(w)dw.
\end{equation}
The final step is to let $n\to \infty$ and then $r\to 0$.  
We see that the denominator $w-iJ_n(w)$ can only vanish at the origin since the sequence of functions $\{J_n\}_n$ is real valued and moreover, it can only vanish in the limit $n\to \infty$. Since the set $E_n\setminus [-c,c]$ is disjoint from the origin, we use dominated convergence and obtain 

\begin{equation} \label{GenCaseLim1}
    \lim_{r\to 0}\lim_{n\to \infty}\int_{E_n\setminus[-c,c]}\frac{1}{w-iJ_n(w)}H_n(w)dw
     =\lim_{r\to 0}\int_{[u(-c_1),u(c_2)]\setminus [-c,c]}\frac{1}{w-iv(u^{-1}(w))}H(w)dw.
\end{equation}

Applying Lemma \ref{DiniLemma}, we see $H_n(w)$ are Dini continuous and converge pointwise to the Dini function $H(w)$.   This implies the integrand on the right side of \eqref{GenCaseLim1} is $L_{\text{loc}}^1$ in a neighborhood about the origin. Therefore, the integration will yield a pair of differences of an $AC$ function evaluated at the bounds \cite{Fol}.  We note that by normalization and a Taylor estimate, $u(-c_1),u(c_2)=o(1)$ as $r\to 0$. Since $0<c<u(c_2)$ we see the same is true for $c$ and therefore the limit $r\to 0$ of \eqref{GenCaseLim1} will be $0$.

Finally, we easily check that second integral in \eqref{GenCaseRed2} meets the conditions of Lemma \ref{LimitLemma} and we obtain    

\begin{equation}
    \lim_{r\to 0}\lim_{n\to \infty}\int_{-c}^c\frac{1}{w-iJ_n(w)}H_n(w)dw=\pi iH(0)=\pi i (\varphi\circ \psi)(0)=\pi i \varphi(0).
\end{equation}

Therefore from \eqref{GenPhin} we have

\begin{equation}
    \Phi^+(0)=\lim_{r \to 0}\lim_{n\to\infty}\Phi(z_n)=\frac{1}{2}\varphi(0)+\frac{1}{2\pi i}\lim_{r \to 0}\int_{\calC \cap \DD_r^c}\frac{\varphi(s)}{s}ds =\frac{1}{2}\varphi(0)+\frac{1}{2\pi i}\text{{\rm{P.V.}}}\int_{\calC }\frac{\varphi(s)}{s}ds
\end{equation}
and the proof is complete. 
%\footnote{An alternative approach is to use the Riemann mapping theorem and map a simply connected domain with boundary containing the curve $\calC$ and interior containing the sequence to the upper half plane.  The conformal map in some sense acts as a parameterization of the curve when it is extended to the boundary of the domain.  The proof is then reduced to the case of an interval.}

\end{proof}
\subsection{Proof of Lemmas}

\begin{proof} [Proof of Lemma \ref{DiniLemma}]
$  $ \newline
\indent Fix $t>0$. Let $t_1,t_2 \in [a,b]$ such that $|t_1-t_2|\leq t$. Consider
\begin{equation} 
\label{compositionOne}
    |\phi \circ \psi (t_1) - \phi \circ \psi (t_2)| = |\phi(\psi(t_1)) - \phi(\psi (t_2))| \leq \omega_{\phi}(|\psi(t_1)-\psi(t_2)|).
\end{equation}
We know that $\displaystyle |\psi(t_2)-\psi(t_1)| = \left|\int_{t_1}^{t_2}\psi'(t)dt\right|\leq \int_{t_1}^{t_2}|\psi'(t)|dt $. Applying the mean value theorem for integrals, there exists $\eta \in (t_1,t_2)$ such that 
\begin{align*}
    |\psi'(\eta)| &= \dfrac{1}{t_2-t_1} \int_{t_1}^{t_2} |\psi'(t)|dt \\
    |\psi(t_1)-\psi(t_2)| &\leq |\psi'(\eta)| |t_2-t_1|.
\end{align*}
Since $\psi$ is continuously differentiable on a closed interval $[a,b]$, therefore, $\psi'$ is bounded, say, by $M$. Without loss of generality, we can assume $M>1$. Therefore,
\begin{equation}
\label{compositionTwo}
  |\psi(t_1)-\psi(t_2)| \leq M |t_1-t_2| \leq Mt.  
\end{equation}
Using \eqref{compositionOne} and \eqref{compositionTwo} and the fact that the modulus of continuity is a non-decreasing function, we get
\begin{equation*}
   |\phi \circ \psi (t_1) - \phi \circ \psi (t_2)| \leq  \omega_{\phi}(|\psi(t_1)-\psi(t_2)|) \leq \omega_{\phi} (Mt)
\end{equation*}
i.e.,
\begin{equation*}
    \omega_{\phi \circ \psi }(t) \leq \omega_{\phi} (Mt).
\end{equation*}
Since $t > 0$ is arbitrary and the above is true for $t=0$ trivially, therefore $\omega_{\phi \circ \psi }(t) \leq \omega_{\phi} (Mt)$ is true for all $t \geq 0$. \\
Now,
\begin{align*}
    \int_{0}^{1} \dfrac{\omega_{\phi \circ \psi }(t)}{t} dt &\leq \int_{0}^{1} \dfrac{\omega_{\phi}(Mt)}{t} dt = \int_{0}^{M} \dfrac{\omega_{\phi}(u)}{u} du \\
    & \leq \int_{0}^{1} \dfrac{\omega_{\phi}(u)}{u} du + \int_{1}^{M} \dfrac{\omega_{\phi}(u)}{u} du .
\end{align*}
Since $\phi$ is Dini continuous, we have
\begin{equation*}
    \int_{0}^{1} \dfrac{\omega_{\phi}(u)}{u} du < \infty.
\end{equation*}
Also, $\phi$ is Dini continuous implies that $\phi$ is continuous and therefore bounded on $K$. Let us say, $|\phi(z)| \leq M_1$ for all $z \in K$. So,
\begin{equation*}
  \int_{1}^{M} \dfrac{\omega_{\phi}(u)}{u} du \leq \int_{1}^{M} \dfrac{\omega_{\phi}(M)}{u} du \leq 2 M_1 \ln(M)  < \infty.
\end{equation*}
Hence proved that $\phi \circ \psi$ is a Dini continuous function. \\
For fixed $t \geq 0$, consider $z_1,z_2 \in K$ such that $|z_1-z_2| \leq t$.
\begin{align*}
    |\chi \phi (z_1) - \chi \phi (z_2)| 
    %&= |\chi(z_1) \phi (z_1) - \chi(z_2) \phi (z_2)| \\
    &= |\chi(z_1)\phi(z_1) - \chi(z_1)\phi(z_2) + \chi(z_1) \phi(z_2) - \chi(z_2) \phi(z_2)| \\
   % &= |\chi(z_1)(\phi(z_1)-\phi(z_2)) + \phi(z_2)(\chi(z_1) - \chi(z_2))| \\
   % & \leq |\chi(z_1)||\phi(z_1)-\phi(z_2)| + |\phi(z_2)||\chi(z_1) - \chi(z_2)| \\
    & \leq |\chi(z_1)| \omega_{\phi}(t) + |\phi(z_2)||\chi(z_1) - \chi(z_2)|
\end{align*}
Since $\chi$ is a continuous function on the compact set $K$ it is bounded on $K$, say by $M_2$. So,
\begin{equation}
\label{productOne}
    |\chi \phi (z_1) - \chi \phi (z_2)| \leq M_2 \omega_{\phi}(t) + M_1|\chi(z_1) - \chi(z_2)|.
\end{equation}
Since $\chi$ is continuously differentiable, its difference quotient will be continuous and therefore bounded when restricted to the compact set $K$.  Let the bound of the difference quotient over $K$ be denoted by $M_3$ from which we have the following estimate:  
\begin{equation*}
    |\chi(z_1) - \chi (z_2)| \leq M_3 |z_1-z_2|\leq M_3 t
\end{equation*}

So, \eqref{productOne} becomes
\begin{equation*}
    |\chi \phi (z_1) - \chi \phi (z_2)| \leq M_2 \omega_{\phi}(t) + M_1 M_3 t.
\end{equation*}
Since $z_1$, $z_2$ are arbitrary points in $K$ such that $|z_1-z_2|\leq t$, therefore, $\omega_{\chi \phi}(t) \leq M_2 \omega_{\phi}(t) + M_1 M_3 t$. Now, using the fact that $\phi$ is a Dini continuous function, we have that $t^{-1} \omega_{\chi \phi}(t) \in L^{1}[0,1]$ i.e. $\chi \phi$ is Dini continuous. 
\end{proof}

\begin{proof} [Proof of Lemma \ref{LimitLemma}]
\begin{enumerate}[label=(\roman*)]
    \item  Without loss of generality, assume $J_n>0$. We write 
\begin{equation}
    \frac{F_n(w)}{w-iJ_n(w)}=\frac{F_n(w)-F_n(0)}{w-iJ_n(w)}+\frac{F_n(0)}{w-iJ_n(w)}.
\end{equation}
Since $J_n(w)$ are real valued, we have $|w|\leq |w-iJ_n(w)|$. Since $F_n$ are Dini, therefore, $F_n(w)-F_n(0)$ are bounded by $\omega_F(|w|)$, the dominated convergence theorem applies therefore giving 

\begin{equation}\label{FirstInt}
    \lim_{n\to \infty}\int _{-c}^c\frac{F_n(w)-F_n(0)}{w-iJ_n(w)}dw=\int _{-c}^c\frac{F(w)-F(0)}{w-iG(w)}dw
\end{equation}
Since the right integrand of \eqref{FirstInt} is $L^1$, we know that as a function of the bound $c$, the integration results in the difference $K(c)-K(-c)$ where $K\in AC([-c,c])$ \cite{Fol}.  In the limit $c\to 0$, this difference also shrinks to zero since absolute continuity implies continuity. 

To obtain the last term, before taking either limit, we deform the contour $[-c,c]$ into  $\gamma_c$ a semi circle in the lower half plane of radius $c$. We evaluate the integral along the deformed contour by first taking the limit $n\to \infty$, using dominated convergence and then parameterizing $\gamma_c$ by $w=c\rm{e}^{it}$ for $t\in[-\pi,0]$.  Additionally, using the hypotheses we know $\lim_{c\to 0}\frac{G(ce^{it})}{ce^{it}}=0$.   The computation results in 

\begin{multline}
    \lim_{c\to 0}\lim_{n\to \infty}\int_{\gamma_c}\frac{F_n(0)}{w-iJ_n(w)}dw\\
    =  \lim_{c\to 0} \int_{-\pi}^0\frac{F(0)}{ce^{it}-ice^{it}\frac{G(ce^{it})}{ce^{it}}}ice^{it}dt=iF(0)\lim_{c\to 0}\int_{-\pi}^0\frac{1}{1-i\frac{G(ce^{it})}{ce^{it}}}dt=\pi i F(0)
\end{multline}

Therefore  \eqref{LemmaEq} holds. 

\item \label{Lemma9ii}
Estimating \eqref{LemmaGenExists} in absolute value, we obtain 
 \begin{multline}\label{LemAbsPV}
     \left |{\rm{P.V.}}\int_{-c}^c \frac{F(w)}{w-iJ(w)}dw\right |=\left|  \lim_{\epsilon\to 0} \int_{[-c,c]\cap \DD_{\epsilon}^c} \frac{F(w)-F(0)}{w-iJ(w)} dw +\int_{[-c,c]\cap \DD_{\epsilon}^c} \frac{F(0)}{w-iJ(w)} dw\right |\\ 
     \leq \lim_{\epsilon\to 0} \int_{[-c,c]\cap \DD_{\epsilon}^c} \frac{\omega_F(|w|)}{|w-iJ(w)|} dw+  \lim_{\epsilon\to 0} \left |\int_{[-c,c]\cap \DD_{\epsilon}^c} \frac{F(0)}{w-iJ(w)} dw\right |
 \end{multline}

Since $J$ is real valued we know $|w-iJ(w)|\geq|w|$ and Dini continuity of $F$ implies existence of the first  limit on the right hand side of \eqref{LemAbsPV}.

For the last term in \eqref{LemAbsPV} we write 

\begin{equation}
    \frac{1}{w-iJ(w)}=\left (\frac{1}{w-iJ(w)}-\frac{1}{w}\right)+\frac{1}{w}=\frac{iJ(w)/w}{w-iJ(w)}+\frac{1}{w}
\end{equation}
Writing $\displaystyle J(w)/w=\int_0^1J'(sw)ds$ and using the Hölder continuity of $J'$ we have the estimate $|J(w)/w|\leq C|w|^\alpha$ with $\alpha\in(0,1]$. Again using the estimate $|w-iJ(w)|\geq|w|$ we therefore have 

\begin{equation}\label{Lem9LastEq}
    \lim_{\epsilon\to 0} \left |\int_{[-c,c]\cap \DD_{\epsilon}^c} \frac{F(0)}{w-iJ(w)} dw\right |\leq |F(0)| \left\{ \lim_{\epsilon\to 0} \left | \int_{[-c,c]\cap \DD_{\epsilon}^c}\frac{1}{w}dw\right|+\lim_{\epsilon\to 0}\int_{[-c,c]\cap \DD_{\epsilon}^c}C|w|^{\alpha-1}dw\right\}.
\end{equation}
The first integral is one of principal value and equal to zero by symmetry and the second exists since $\alpha\in (0,1]$. 
\item Starting from \eqref{LemAbsPV} and considering this as a function of $c$, we follow the equality and inequality chains.  At the end of the two chains, the integrals that do not vanish i.e. the first term in the right most inequality of \eqref{LemAbsPV} and the last term on the right side of \eqref{Lem9LastEq} both have integrands that belong to $L^1[-1,1]$, by the fundamental theorem of Lebesgue integrals \cite{Fol}  these integrals are absolutely continuous functions of $c$  on $[0,1]$.  Therefore the limit $c\to 1$ exists and the argument is equivalent to \ref{Lemma9ii} of this Lemma.

\end{enumerate}

\end{proof}

\end{document}